\documentclass[onecolumn]{IEEEtran}
\usepackage{amsmath,amssymb,amsthm,mathtools}
\mathtoolsset{showonlyrefs}
\usepackage{cite}
\usepackage{graphicx}
\usepackage[hidelinks]{hyperref}
\usepackage{tikz}
\usetikzlibrary{arrows.meta,decorations.pathmorphing}
\usepackage{enumerate}

\usepackage{setspace}
\usepackage{verbatim}

\newcommand{\R}{\mathbb{R}}
\newcommand{\GOE}{\operatorname{GOE}}

\newcommand{\avg}[1]{\left\langle #1\right\rangle}

\newtheorem{theorem}{Theorem}
\newtheorem{lemma}[theorem]{Lemma}
\newtheorem{proposition}[theorem]{Proposition}
\newtheorem{corollary}[theorem]{Corollary}
\newtheorem{remark}[theorem]{Remark}

\title{Gain of Entrainment in Nonlinear Cascades }

\author{Ram~Massas and Michael~Margaliot%
\thanks{RM and MM are with the School of Electrical Engineering,
Tel Aviv University, Tel Aviv 69978, Israel.
Corresponding author: Michael Margaliot (michaelm@tauex.tau.ac.il).}%
\thanks{The research of MM  was partially supported by the Israel Science
Foundation under Grant 221/24.}}
\usepackage{amsthm}
\newtheorem{example}[theorem]{Example}

\begin{document}

\maketitle

\begin{abstract}
We consider the gain of entrainment (GOE)--the difference between the average steady-state output under a periodic input and the steady-state output under a constant input with the same mean--for an $n$-stage feedforward cascade of stable first-order filters interleaved with static nonlinearities. The main result is an exact decomposition of~GOE as a weighted sum of local Jensen gaps, where each gap quantifies the mean shift generated by a  nonlinearity, and each weight is a product of downstream incremental  gains divided by linear time constants. We provide a Bregman-divergence interpretation of the decomposition, and a  second-order small-amplitude of GOE separating local curvature, fluctuation energy, and differential gains.
We demonstrate the theoretical results using a Michaelis–Menten cascade showing that any nonconstant periodic feeding strictly reduces the average terminal product relative to constant feeding with the same mean.
\end{abstract}


\section{Introduction}
\label{sec:introduction}

Many natural and artificial systems and processes are subject to $T$-periodic excitation or forcing, and 
their proper functioning requires synchronization or entrainment to a $T$-periodic  behavior pattern. Synchronous generators must entrain to the frequency of the grid. Internal clocks in biological organisms must entrain to the 24-hour solar day.  
Phase-locked loops synchronize to periodic reference signals, and
cardiac pacemaker cells synchronize to periodic electrical stimulation.

Entrainment is by no means a generic property of nonlinear systems~\cite{NIKOLAEV20181232,takac_sub_harmonics}.
However, there are important classes   of dynamical systems that do  entrain to periodic excitations. These include contractive systems~\cite{AminzareSontag2014,Russo2010}, totally positive differential systems~\cite{TPDS_MAR_SONATG},
and monotone dynamical systems with a first integral~\cite{Ji-Fa-periodic}.

Entrainment raises the question of whether periodic forcing can do more than
induce synchronization, namely, whether it can also enhance system performance.
The concept of \emph{gain of entrainment}~(GOE) allows  addressing  this question in a rigorous manner. To explain this, consider the nonlinear dynamical system 
\begin{align} \label{eq:nonlinsys}
\dot x&=f(x,u), \nonumber \\
y&=h(x,u),
\end{align}
with state
vector \(x\), input \(u\), and a scalar performance output~$y$.
The system is said to \emph{entrain} if for any \(T\)-periodic input \(u\)  any solution of~\eqref{eq:nonlinsys} converges to a unique~$T$-periodic limit cycle~$\gamma^u$, and thus the output converges to  the $T$-periodic output
\[
y^\gamma(t):=h(\gamma^u(t),u(t)).
\]
Suppose that the objective is to maximize the output. Since this converges to~$y^\gamma$, it is natural to consider the time average
\[
\avg{y^\gamma} := \frac{1}{T}\int_0^T h(\gamma^u(t),u(t))\,dt.
\]
As an illustration, the output may represent the average traffic flow through a
road network regulated by periodically varying traffic lights, and the goal is   
to maximize throughput. As another example, the output may represent the product yield of a chemical reactor operated under   periodic     feeding.

Let
\[
\avg{u}:=\frac1T\int_0^T u(t)\,dt
\]
denote the mean value of the input.
Since the system entrains, when  the constant input \(\avg{u}\) is applied to~\eqref{eq:nonlinsys}, the state converges to an
equilibrium,  denoted~\(e^{\avg{u}}\), and  thus the output converges to  the constant value
\[
h(e^{\avg{u}},\avg{u}).
\]
The gain of entrainment associated with the $T$-periodic input \(u\) is then defined
by
\[
\GOE(u):=
\frac1T\int_0^T h(\gamma(t),u(t))\,dt
-h(e^{\avg{u}},\avg{u}).
\]
Hence, \(\GOE(v)>0\) means that redistributing the input periodically in time
improves the average output. Conversely,
\(\GOE(v)<0\) means that the temporal variation is detrimental relative to
constant operation, so in this  case entrainment allows 
synchronization to periodic signals, but at the cost of decreasing  the output, on average.

For stable linear dynamics with a linear output, the average periodic state
coincides with the equilibrium associated with the average input. Therefore,
the corresponding~GOE is zero. In nonlinear systems, the situation is more interesting. 
To demonstrate this, we consider
an example from Ref.~\cite{Pavlov2007}, that  studied the frequency response of convergent systems (that are closely related to contractive systems~\cite{RUFFER2013277}).  
\begin{example}\label{exa:nonlin_with_goe}
    Consider the non-linear system:    \begin{align}\label{eq:nolonpav}
    \dot x_1&=-x_1+x_2^2,\nonumber\\
    \dot x_2&=-x_2+u,\nonumber\\
    y&=x_1.
    \end{align}
This is the series interconnection of two (scalar)  contractive  systems,  and is thus a  contractive system (see, e.g.,~\cite{ofir2021sufficient}).
We compare the effect of two controls:
\begin{enumerate}
    \item  $u_1(t)=a\sin(\omega t)$, 
where~$  a,\omega>0$; and
\item 
$u_2(t)=0$.
 \end{enumerate}
Note that $u_1$ is~$T$-periodic with~$T:=2\pi/\omega$,
and that~$u_2(t)=\avg{u_1}$.
For the zero control  the state vector converges  to $e^{\avg{u_1}}=0$, so~$y(t)$ converges to zero. 
For the control~$u(t)=u_1(t) $,   a calculation shows that the state vector 
 converges to a $T$-periodic solution~$\gamma^{u_1}$ satisfying 
\begin{equation*} 
s(\omega)
\gamma_1^{u_1}(t)= 4a^2\omega^2+a^2-2a^2\omega\sin(2t\omega-p(\omega))-a^2\cos( 2t\omega-p(\omega)),  
\end{equation*}
with~$s(\omega):=2(4\omega^4+5\omega^2+1) $ and~$p(\omega):=2\,
 \arctan (\omega)$,
so 
$\frac{1}{T}\int_0^T \gamma_1^{u_1}
(t) \mathrm{d} t 
 = \frac{a^2 }{2(1+\omega^2)} 
$. Thus, 
\begin{align}\label{eq:GOE_ISQUAD}
\GOE(u_1)=\frac{a^2 }{2(1+\omega^2)} ,
\end{align}
which is  positive and depends on both  the excitation amplitude~$a$, and excitation
frequency~$\omega$. 

Note that $\GOE(u_1)$
decreases to zero as $\omega\to\infty$, reflecting the attenuation of high-frequency inputs by the stable
linear filter preceding the quadratic stage.
For a fixed $\omega$, the quadratic dependence of $\GOE(u_1)$ on the control amplitude~$a$ implies the following. If $a\ll 1$ [$a\gg 1$] then $\GOE(u_1)\ll a$ [$\GOE(u_1)\gg a$]. 
\end{example}

Since constant inputs form a special case of periodic inputs, one might expect
that~GOE is always nonnegative, yet this intuition is generally
incorrect. 
Analyzing~GOE is a challenging theoretical
problem 
because in general  nonlinear systems both~$\gamma^u$ and~$e^{\avg{u}}$ are not known explicitly. Maximizing~GOE can   be posed as  an optimal control problem with a constraint on the average control, but 
explicit solutions are known only in relatively simple settings, such as scalar
systems~\cite{rapo_convex}.

  GOE is important in 
  numerous applications 
 including periodically harvested fisheries, photobioreactors operated  by
time-varying light, traffic networks controlled by periodic signal timing,  
drug delivery protocols based on periodic  dosing schedules, gene expression  under periodic  regulation, and
manufacturing systems operating under cyclic production schedules.

In models of
protein translation, one may ask whether periodic variation of initiation or
elongation resources may increase the average production rate.  
Ref.~\cite{GOE_RFM_2023}  considered GOE
in an important nonlinear model from  systems biology called the ribosome flow model~(RFM). 
In this case, the scalar output corresponds to  the protein production rate in mRNA translation,
the $n$-dimensional  state vector corresponds to the densities of ribosomes along~$n$ different sections of the mRNA molecule, and 
and the~$(n+1)$-dimensional  input to the transition rates of ribosomes along different segments  of the mRNA molecule. The RFM is an (almost) contractive system~\cite{weak_contract}, and thus entrains to periodic  inputs~\cite{Margaliot2014Entrain}. 
It was shown in~\cite{GOE_RFM_2023} 
that in several special cases~$\GOE(u)\leq 0$. More recently, and using a different technique, Ref.~\cite{cost_of_entrain} proved that in the RFM we always  have~$\GOE(u)\leq 0$, with equality only in  the degenerate case 
where all the $T$-periodic transition rates are equal, up to multiplication by a positive constant.
Thus,~$\GOE(u)<0$
for any non-trivial periodic modulation, implying  that entrainment to periodic controls always comes with a cost in terms of the protein production rate. 

Ref.~\cite{Katz2024} 
  considered a class of 
  weakly contractive systems under a control~$u(t)=a+\varepsilon \tilde u(t)$, with~$\tilde u$ a $T$-periodic control and~$\varepsilon\in \R$.
  Using   a first-order expression
for the GOE, it was shown   that GOE is
inherently a second- or higher-order effect in~$\varepsilon$.
Ref.~\cite{MassasMargaliot_linear_sys_nonlin_output} investigated systems with
linear dynamics and a static nonlinear output map~$h$, establishing a connection between the
sign of~GOE and the convexity or concavity of~$h$.

Here, we consider   a   cascade (or serial connection) of scalar linear systems with a nonlinear output. Feedforward cascades are a common interconnection structure in engineering, biology, chemistry, communication systems,  and economics.
Understanding how local nonlinearities combine to determine global performance remains a fundamental challenge. 
We analyze GOE of such a cascade. 
This system is contractive, and thus entrains to periodic excitations.   
In a cascade with several nonlinear
stages, the situation is fundamentally different from that of  a single linear system with a nonlinear output~\cite{MassasMargaliot_linear_sys_nonlin_output},
as every nonlinear
stage may create its own mean displacement, and each displacement must then
propagate through all downstream stages before reaching the output.
A downstream stage may amplify or attenuate this displacement and may even reverse its sign. Consequently, the sign and magnitude
of the overall GOE cannot generally be determined by examining the curvature
of a single function.

 The main contributions of this paper are as follows. We derive an exact decomposition of  GOE in the cascade as  a sum of local contributions associated with each nonlinear stage and propagation weights determined by the downstream dynamics. This decomposition separates the generation of average shifts due to nonlinearity from their propagation through the cascade, providing a transparent stage-by-stage interpretation of GOE. We further establish sufficient conditions that determine the sign of the GOE, derive an equivalent representation based on Bregman divergences, and develop a second-order approximation that relates the GOE to the local curvature of the nonlinearities and the differential gains of the cascade. We illustrate the theoretical results  using  a cascade with Michaelis--Menten kinetics, demonstrating how the proposed framework explains the effect of periodic forcing on the average output.
 
For a $T$-periodic  function $q$,   let 
\[ 
    \avg q
    :=
    \frac{1}{T}\int_0^T q(t)\,\mathrm{d}t.
\]
Note that if~$q$  is $C^1$ then
 \begin{align}
     \label{eq:avgderiszero}
 \avg{\dot q}=\frac{1}{T}\int_0^T \dot q(t) \mathrm{d}t =\frac1{T}(q(T)-q(0) )=0. 
\end{align}

\section{Main Results}\label{sec:main}
Consider the   $n$-stage feedforward cascade
\begin{align} \label{eq:cascade-general}
    \dot x_1(t)
    &=
    -a_1x_1(t)+\phi_1\bigl(u(t)\bigr),
    \nonumber\\
    \dot x_i(t)
    &=
    -a_ix_i(t)+\phi_i\bigl(x_{i-1}(t)\bigr),
    \qquad i=2,\ldots,n,
    \\
    y(t)
    &=
    x_n(t),\nonumber 
\end{align}
where~$a_i>0$  for all~$i$, and every~$\phi_i:\R\to\R$ is a continuous   function (and thus maps compact  sets to compact  sets). 
Each state is thus  a stable first-order dynamical filter driven by a
static nonlinear map of the signal arriving from the preceding
stage.

Consider a continuous $T$-periodic   control~$v$. Then~$v$ takes values in a compact set, and thus so does~$\phi_1(v)$. This implies that~$x_1 $ is bounded and converges to a unique~$T$-periodic limit cycle~$\gamma_1^v$. Using this implies that~$x_2 $ is bounded and converges to a
unique~$T$-periodic limit cycle~$\gamma_2^v$. Continuing in this fashion, we have that for any initial condition~$x(0)$ the state~$x$ converges to a 
unique~$T$-periodic limit cycle~$\gamma^v$.

For the constant 
control~$u=\avg{v}$, the state 
$x$ converges to an  equilibrium $e^{\avg{v}}$ with 
\begin{align*}
e_1^{\avg{ v}}& = a_1 ^{-1} \phi_1(  \avg {v}),\nonumber\\
e_i^{\avg{ v} }&= a_i ^{-1} \phi_i (
 e_{i-1}^{\avg{v}}) ,
\quad   i=2,\dots,n.
\end{align*}

Since the output is $y=x_n$, we get that   
\begin{equation}
    \GOE(v)
    =
    \frac{1}{T}
    \int_{0}^{T}\gamma_n^v(t)\,\mathrm{d}t
    -
    e_n^{\avg{ v}}
    =
    \avg{\gamma_n^v}
    -
    e_n^{\avg{ v}}.
    \label{eq:intro-cascade-goe}
\end{equation}

 Note that for the cascade \eqref{eq:cascade-general} 
 it is possible in general to give an explicit \emph{integral representation} for each $\gamma_i^v$, but not   a closed-form elementary expression.

\subsection{Formula for GOE}
Define the mean displacement at stage $i$ by
\begin{equation}
    \Delta_i(v)
    :=
    \avg{\gamma_i^v}-e_i^{\avg{ v}},
    \qquad i=1,\ldots,n.
    \label{eq:delta-def}
\end{equation}
By \eqref{eq:intro-cascade-goe},
\begin{equation}
    \GOE(v)=\Delta_n(v).
    \label{eq:goe-delta}
\end{equation}

It is useful to define
\begin{equation}\label{eq:gamm0}
\gamma_0^v(t) := v(t).
\end{equation}
To identify the     contributions of the nonlinear functions, define
\begin{align}  \label{eq:jensen-internal}
    J_i(v)
    &:=
    \avg{\phi_i(\gamma_{i-1}^v)}
    -
    \phi_i\bigl(\avg{\gamma_{i-1}^v}\bigr),
    \qquad i=1,\ldots,n.
\end{align}
Thus, $J_i$ is the Jensen gap generated at stage $i$ by $\phi_i$.
For every internal stage $i=2,\ldots,n$, define the incremental
gain $  L_i(v)$ by
\begin{equation}\label{eq:secant-gain} 
    L_i(v)
    :=
    \dfrac{
    \phi_i(\avg{\gamma_{i-1}^v})
    -
    \phi_i(e_{i-1}^{\avg {v}})
    }{
    \Delta_{i-1}(v)
    },
\end{equation} 
if 
$ \Delta_{i-1}(v)\neq0$,
and~$L_i(v)=0$, otherwise.
This   ensures that
\begin{equation}\label{eq:ensuresthat}
    \phi_i(\avg{\gamma_{i-1}^v})
    -
    \phi_i(e_{i-1}^{\avg {v}})
    =
    L_i(v)\Delta_{i-1}(v)
\end{equation}
also when $\Delta_{i-1}(v)=0$.
If~\eqref{eq:secant-gain} is well-defined also when~$\Delta_{i-1}(v)=0$ then   the definition~\eqref{eq:secant-gain} is used for all~$v$. 
 Our first main result provides a formula for~$\GOE(v)$ 
as a weighted sum of~$J_k(v)$, $k=1,\dots,n$.
We can now state our first main result. 
\begin{theorem}[Formula for GOE of  a cascade]
\label{thm:exact-decomp}
For every continuous $T$-periodic input $v$, we have 
\begin{equation}  \label{eq:exact-goe}
     \GOE(v)
    =
    \sum_{k=1}^n W_k(v)J_k(v)
    ,
\end{equation}
where
\begin{equation}
    W_k(v)
    :=
    {a_k} ^{-1}
    \prod_{j=k+1}^n
 {a_j}  ^{-1} {L_j(v)} ,
    \qquad k=1,\ldots,n, 
    \label{eq:cascade-weight}
\end{equation}
with the convention that an empty product equals~$1$.
\end{theorem}

Eq.~\eqref{eq:exact-goe} separates two mechanisms determining~GOE. The \emph{local term}
$J_k$ measures the mean shift  created  by the curvature of $\phi_k$,
whereas the weight~$W_k$ measures how this shift is \emph{propagated} from stage~$k$
through the downstream stages~$k+1,\ldots,n$ to the output. In particular,
a negative downstream   gain reverses the sign of every contribution
that passes through it. 
 
\begin{proof}
Averaging the first equation of the cascade along the periodic orbit gives
$
    0
    =
    -a_1\avg{\gamma_1^v}
    +
    \avg{\phi_1(v)}.
$
Subtracting the identity 
$    0
    =
    -a_1e_1^{\avg {v}}
    +
    \phi_1(\avg {v})
$ yields
\begin{equation}
    \Delta_1(v)
    =a_1^{-1}
    {J_1(v)} .
    \label{eq:delta-first}
\end{equation}
For $i=2,\ldots,n$, periodic averaging and subtraction of the corresponding
equilibrium equation give
\begin{align*}
    a_i\Delta_i(v)
    &=
    \avg{\phi_i(\gamma_{i-1}^v)}
    -
    \phi_i(e_{i-1}^{\avg {v}})
    \nonumber\\
    &=
    J_i(v)
    +
    \phi_i(\avg{\gamma_{i-1}^v})
    -
    \phi_i(e_{i-1}^{\avg {v}})
    \nonumber\\
    &=
    J_i(v)+L_i(v)\Delta_{i-1}(v), 
\end{align*}
where the last step follows from \eqref{eq:ensuresthat}. Thus,  
\begin{equation*}
    \Delta_i(v)
    =a_i^{-1} \left (
    J_i(v)
    +
    L_i(v)\Delta_{i-1}(v) \right).
    \label{eq:delta-recursion2}
\end{equation*}
Starting from \eqref{eq:delta-first}, expanding the recursion forward to
$i=n$, and using \eqref{eq:goe-delta} yields \eqref{eq:exact-goe}.
\end{proof}

\begin{example}
    Consider   again system~\eqref{eq:nolonpav} in Example~\ref{exa:nonlin_with_goe}. 
Define
   $ \xi_1:=x_2$ and
    $\xi_2:=x_1$.
Then  
\begin{align}\label{eq:samae}
    \dot \xi_1(t)
    &=
    -\xi_1(t)+u(t),\nonumber\\
    \dot \xi_2(t)
    &=
    -\xi_2(t)+\xi_1^2(t),\nonumber\\
    y(t)&=\xi_2(t),
\end{align}
which is   of the form
\eqref{eq:cascade-general} with~$n=2$, 
$\phi_1(s)=s$, $\phi_2(s)=s^2$, and~$a_1=a_2=1$.
In this case,
\[
J_1(u)=  \avg{\phi_1(u)}-\phi_1(\avg{ u})=0,
\]
and~\eqref{eq:exact-goe} 
gives
\begin{align*}
    \GOE (u) &= W_1(u) J_1(u)+W_2(u) J_2(u) \\
          &=     J_2(u)\\
          &=
       \avg{\phi_2(\gamma_{1}^u)}
    -
    \phi_2\bigl(\avg{\gamma_{1}^u}\bigr)\\
    & =   \avg{ (\gamma_{1}^u)^2}
    -
 \avg{\gamma_{1}^u} ^2
    .
\end{align*}

Let
    $u_1(t)=a\sin(\omega t)$, with~$a,\omega>0$. This is  $T$-periodic with~$T:=2\pi/\omega$. 
 The corresponding 
 unique periodic solution of~\eqref{eq:samae} 
 satisfies
\begin{align*}
    \gamma_1^{u_1}(t)
    &=
    \frac{a}{1+\omega^2}\sin(\omega t)
    -
    \frac{a\omega}{1+\omega^2}\cos(\omega t).
    \label{eq:gamma2-positive}
\end{align*}
This gives  $\avg{\gamma_1^{u_1}}=0$,   
and
  $  \avg{(\gamma_1^{u_1})^2}
    =    \frac{a^2}{2(1+\omega^2)}$,
    so
\begin{equation*} 
    \GOE(u_1)
    =
    \frac{a^2}{2(1+\omega^2)}
     ,
\end{equation*}
and this agrees with~\eqref{eq:GOE_ISQUAD}. 
\end{example}

\begin{remark}
    If for some index $k$ the function $\phi_k$   is affine then \eqref{eq:jensen-internal} implies that
$J_k(v) =0 $ and Thm. \ref{thm:exact-decomp} implies that the $k$th stage in the cascade    contributes  no local Jensen gap to GOE (but it may still transmit, amplify, attenuate, or reverse upstream contributions).
\end{remark}

An important consequence of Theorem~\ref{thm:exact-decomp} is that it yields simple sufficient conditions guaranteeing the sign of the GOE in a cascade. Notably, these conditions can be verified without computing the equilibrium nor the periodic limit cycle. The following example demonstrates this.

\begin{corollary}
\label{cor:sign}
Suppose that  the functions~$\phi_1,\ldots,\phi_n$ are convex [concave], and that the functions 
$\phi_2,\ldots,\phi_n$ are nondecreasing. Then
    $\GOE(v) \geq 0$
[$\GOE(v)\leq 0$] for every continuous periodic input~$v$. 
\end{corollary}

\begin{proof}
We prove the result for the case where all the~$\phi_k$s are convex.   Jensen's inequality 
gives~$J_k(v)\geq0$ for every~$k$.
The assumption that $\phi_2,\ldots,\phi_n$ are nondecreasing gives   
  $L_j(v)\geq0$ for $j=2,\ldots,n$, so all the weights
$W_k(v)$ are nonnegative. Now~\eqref{eq:exact-goe}  gives~$\GOE(v)\geq0$.
\end{proof}

In general, convexity [concavity]  properties of the~$\phi_k$s are not enough to deduce the sign of GOE, due to the effect of the~$L_j$s.  
\begin{example}[Convex $\phi_i$s yet negative GOE]
\label{exa:negative}
Consider the two-stage cascade
\begin{align*}
    \dot x_1
    &=
    -x_1+v^2,
   \\
    \dot x_2
    &=
    -x_2-x_1,\\
 y&=x_2, 
\end{align*}
which is   of the form
\eqref{eq:cascade-general} with~$n=2$, 
$\phi_1(s)=s^2$, $\phi_2(s)=-s$, and~$a_1=a_2=1$.
 Note that both $\phi_i$s
are convex. 
In this case, $\phi_2$ is affine, so~$J_2=0$
and~\eqref{eq:exact-goe} 
gives
\begin{align*}
    \GOE (u) &= W_1(u) J_1(u)+W_2(u) J_2(u) \\
          &=     L_2(u) J_1(u)\\
          &=-J_1(u).
\end{align*}

Consider the control
$
    v(t)=a\sin(\omega t),
 $ with~$    a,\omega>0$.  Then 
 $
    J_1(v)
    =
    \avg{v^2}
    =
    a^2/2$,
    so $
    \GOE(v)
    =
    - {a^2}/{2}
    <0$.
Thus, the positive Jensen gap created by the strictly convex upstream map~$\phi_1$ is sign-reversed by the decreasing downstream map~$\phi_2$.
\end{example}

The proof of 
Corollary~\ref{cor:sign} 
suggests simple conditions guaranteeing  strict positivity [negativity] of~GOE. The next remark demonstrates this.

\begin{remark}[Strict positivity of GOE]
\label{rem:strict}
Suppose that the functions~$\phi_1,\ldots,\phi_n$ are convex,   that the functions 
$\phi_2,\ldots,\phi_n$ are nondecreasing, 
and that there
 exists an index $k$ such that
   $ J_k(v)>0$, 
and
 $   L_j(v)>0$,
    for~$ j=k+1,\ldots,n$.
Then $\GOE(v)>0$. 
(Note that~$J_k(v)>0$ whenever~$\phi_k$ is strictly convex on the range of
its periodic argument,
and that argument is nonconstant.)
\end{remark}

\subsection{Bregman-Divergence Representation}
Recall that if~\(g:\mathbb R^n\to\mathbb R\) is $C^1$  
then the function~$D_g:\R^n\times\R^n\to\R$
defined by
\begin{equation}\label{eq:d_geq}
D_g(x,z):=
g(x)-g(z)-(\nabla g(z))^\top (x-z) 
\end{equation}
is called the
Bregman divergence associated with~\(g\)~\cite{bregman1967}.
Geometrically, $D_g(x,z)$
  measures how much the graph of $g$ lies above its tangent plane at~$z$. 
The function~$D_g$
     is generally
neither symmetric nor a metric, but its sign is directly related to the
curvature of the underlying function. In particular, it is nonnegative for
convex functions,  and nonpositive for concave functions~\cite{bregman1967}.

Suppose now that~$\phi:\R\to\R$ is~$C^1$, and 
let $u$ be a scalar $T$-periodic control.
Fix $t\in[0,T)$. Then
\begin{align*}
D_\phi(u(t),\avg{u})&=
\phi(u(t))-\phi(\avg{u})
-\phi'(\avg {u})(u(t)-\avg{u}),
\end{align*}
so
\begin{align*}
\avg{D_\phi(u(t),\avg{u}) }&=
\avg {\phi(u)}-\phi(\avg{u}).
\end{align*}
Comparing this with \eqref{eq:jensen-internal} and using Theorem~\ref{thm:exact-decomp} yields the following result.

\begin{corollary}
[Geometric decomposition for GOE.]
\label{prop:bregman} 
Suppose  that all the~$\phi_i$s are~$C^1$, and
define the local average Bregman terms by
\begin{align}\label{eq:big_b}
    B_i(v)
    &:=
    \avg{
    D_{\phi_i}
    \bigl(
    \gamma_{i-1}^v,
    \avg{\gamma_{i-1}^v}
    \bigr)
    },
    \qquad i=1,\ldots,n.
 \end{align}
Then
\begin{equation}
    \GOE(v)
    =
    \sum_{k=1}^n W_k(v)B_k(v).
    \label{eq:bregman-goe}
\end{equation}
\end{corollary}

Thus,  every stage creates a local geometric discrepancy, and the
downstream part of the cascade propagates  this discrepancy to the output
through the weight $W_k$.

\subsection{Small-Amplitude Periodic Excitations}
\label{sec:small-amplitude}

We next study the local behavior of GOE in \eqref{eq:cascade-general} under small periodic
perturbations of a constant control.   
Consider the control 
\begin{equation}
    v_\varepsilon(t)
    =
    a +\varepsilon w(t),
    \label{eq:small-input}
\end{equation}
where $w$ is continuous and     $T$-periodic with $ \avg{w}=0$, and $\varepsilon\in\R$. Thus,~$\avg{v_\varepsilon}=a$. 
To simplify the notation, let 
\begin{equation*}
    \gamma^\varepsilon(t)
    :=
    \gamma^{v_\varepsilon}(t)
\end{equation*}
denote the unique $T$-periodic solution of the nonlinear cascade
corresponding to $v_\varepsilon$. When $\varepsilon=0$, the control is constant
and the periodic solution reduces to a unique  equilibrium:
\begin{equation*}
    \gamma^0(t)=e^{a}.
\end{equation*}

We assume throughout this section that every $\phi_i$ is $C^3$ on a
neighborhood of the relevant equilibrium and periodic limit cycle. 
Smooth
dependence of the periodic solution on $\varepsilon$ yields the uniform
expansion
\begin{equation}
    \gamma_i^\varepsilon(t)
    =
    e_i^a+\varepsilon r_i(t)
    +\varepsilon^2s_i(t)
    +O(\varepsilon^3),
    \qquad i=1,\ldots,n,
    \label{eq:orbit-expansion}
\end{equation}
with
\begin{align}\label{eq:defrisi}
    r_i(t)
    :=
    \left.
    \frac{\partial } 
    {\partial\varepsilon}\gamma_i^\varepsilon(t)
    \right|_{\varepsilon=0},
\quad
    s_i(t)
    :=
    \frac{1}{2}
    \left.
    \frac{\partial^2 } 
    {\partial\varepsilon^2}\gamma_i^\varepsilon(t)
    \right|_{\varepsilon=0}.
\end{align}
%
Note that since $\gamma_ i ^\varepsilon(t)$ is $T$-periodic, so are $r_i (t) $ and $s_i(t)$. 

It is useful to denote
\[
e^a_0:=a.
\]

We can now state our second main result. 
\begin{proposition}[GOE to second-order]
\label{prop:second-order}
Consider~\eqref{eq:cascade-general} with the control   \eqref{eq:small-input}. Then 
\begin{equation}
    \GOE(v_\varepsilon)
    =
    \frac{\varepsilon^2}{2}
    \sum_{k=1}^n
    \rho_k \phi_k''(e_{k-1}^a) c_{k-1}
    +
    O(\varepsilon^3),
    \label{eq:second-order-goe}
\end{equation}
where
\begin{align}\label{eq:defck}
    \rho_k
   & :=
    a_k^{-1}
    \prod_{j=k+1}^n a_j^{-1} \phi_j'(e_{j-1}^a) ,\nonumber \\
     c_k&:= \begin{cases}
     \avg{ w^2 }, & \text {if }      k=0,\\
     \avg{r_{k }^2} , & \text {if }      k>0. 
     \end{cases}
 \end{align}
\end{proposition}

Eq. \eqref{eq:second-order-goe} shows  that  three mechanisms  
determine the leading-order GOE. At stage $k$, $\phi_k''(e_{k-1}^a)$ is the
local curvature of the nonlinear map, $c_{k-1}$ is the mean-square amplitude
of the first-order fluctuation incident on that map, and $\rho_k$ is the
differential gain through which the resulting mean displacement propagates
from stage $k$ through stages $k+1,\ldots,n$ to the output.

\begin{proof}
Since the output is $y=x_n$,  we have 
\begin{align}\label{eq:goe_pertu}
    \GOE(v_\varepsilon)
    &=
    \avg{\gamma_n^\varepsilon}-e_n^a
    \nonumber\\
    &=
    \varepsilon\avg{r_n}
    +
    \varepsilon^2\avg{s_n}
    +
    O(\varepsilon^3),
\end{align}
so we need to compute~$\avg{r_n}$ and~$\avg{s_n}$. 
A Taylor expansion gives  
\begin{align}\label{eq:expa1}
    \phi_1( v_\varepsilon)
    &= \phi_1(a+\varepsilon w  ) \nonumber\\
    &=\phi_1(a)+\varepsilon \phi_1'(a) w  +
    \frac{\varepsilon^2} {2}   \phi_1'' (a)     w^2 
    +
    O(\varepsilon^3)  , 
\end{align}
and for~$i>1$, we have 
\begin{align}
    \phi_i(\gamma_{i-1}^\varepsilon)
    &= \phi_i( e_{i-1}^a+\varepsilon r_{i-1} 
    +\varepsilon^2 s_{i-1} 
    +O(\varepsilon^3) ) \nonumber\\&=
    \phi_i(e_{i-1}^a)
    +
    \varepsilon \phi_i'(e_{i-1}^a)  r_{i-1}
    \nonumber\\
    &\quad
    +
    \varepsilon^2
    \left(
    \phi_i'(e_{i-1}^a) s_{i-1}
    +
    \frac{\phi_i''(e_{i-1}^a)}{2}r_{i-1}^2
    \right)
    +
    O(\varepsilon^3).
    \label{eq:nonlinear-local-expansion}
\end{align}
Substituting~\eqref{eq:orbit-expansion}, \eqref{eq:expa1},    and~\eqref{eq:nonlinear-local-expansion} in~\eqref{eq:cascade-general} and equating the
coefficients of $\varepsilon$ yields the first-order variational cascade
\begin{align}
    \dot r_1
    &=
    -a_1r_1+\phi_1'( a) w,
     \nonumber\\
    \dot r_i
    &=
    -a_ir_i+\phi_i'(e_{i-1}^a)  r_{i-1},
    \qquad i=2,\ldots,n.
    \label{eq:first-general}
\end{align}
 Averaging these
equations over one period gives
\begin{align*}
    \avg{r_1}
    &=
   a_1^{-1} \phi_1'( a) \avg{w}=0,\\
    \avg{r_i}
    &=
  a_i^{-1}  \phi_i'(e_{i-1}^a) \avg{r_{i-1}},
    \qquad i=2,\ldots,n,
\end{align*}
and  iterating 
this gives 
\begin{align}\label{eq:avg_ri}
    \avg{r_i}=0,\quad 
    i=1,\ldots,n.
\end{align}

Equating the coefficients of $\varepsilon^2$ gives the second-order
variational equations
\begin{align}
    \dot s_1
    &=
    -a_1s_1+\frac{\phi_1''(a)}{2} w ^2 ,
   \nonumber \\
    \dot s_i
    &=
    -a_is_i+ \phi_i'(e_{i-1}^a) s_{i-1}
    +\frac{\phi''_i(e_{i-1}^a)}{2}r_{i-1}^2,
    \qquad i=2,\ldots,n.
    \label{eq:second-general}
\end{align}
Averaging \eqref{eq:second-general}   over one period gives
\begin{align*}
    \avg{s_1}
  &  =
\frac1{2}   a_1^{-1}  \phi_1''( a)  c_0   , \\
    \avg{s_i} &  =
  a_i^{-1}   \phi_i'(e_{i-1}^a) \avg{s_{i-1}}
    +
\frac1{2}  a_i^{-1}  \phi_i''(e_{i-1}^a) c_{i-1} ,
    \qquad i=2,\ldots,n.
    \end{align*}
Iterating  this recursion from stage $1$ to the stage~$n$  gives 
\begin{equation*}
    \avg{s_n}
    =
    \frac{1}{2}
    \sum_{k=1}^n\rho_k   \phi_k''(e_{k-1}^a) c_{k-1}.
    \label{eq:snmean}
\end{equation*}
Substituting this and~\eqref{eq:avg_ri} in~\eqref{eq:goe_pertu} completes the proof of  Prop. \ref{prop:second-order}. 
\end{proof}

\section{An Application: Periodic Substrate Feeding in an Enzymatic
Michaelis--Menten Cascade}
\label{sec:enzymatic-cascade}
In biochemical and enzymatic processes,   a substrate may be supplied as a constant feeding
or according to a periodic
feeding schedule, with both protocols having the same mean concentration.
Because enzymatic reaction rates are typically nonlinear and saturating,
these two feeding strategies need not generate the same average terminal
product. 
A natural question is whether periodic modulation may outperform  the constant one. 

We now apply the theoretical results derived above to analyze GOE in an open enzymatic cascade with
Michaelis--Menten kinetics, that is, 
\begin{align}
    \dot x_1(t)
    &=
    -a_1x_1(t)
    +
    \frac{V_1v(t)}{K_1+v(t)},
    \nonumber\\
    \dot x_i(t)
    &=
    -a_ix_i(t)
    +
    \frac{V_ix_{i-1}(t)}
    {K_i+x_{i-1}(t)},
    \quad i=2,\ldots,n,
    \label{eq:mm-cascade-internal}\\
    y(t)&=x_n(t),
    \nonumber 
\end{align}
with $
    a_i,    V_i,
    K_i>0$.
Here,     $v(t)\geq0$ is  the concentration of an externally supplied
substrate or biochemical activator.
The constant
 $a_i$ describes  the first-order degradation, dilution, or deactivation rate
at stage $i$, $V_i$ is the limiting enzymatic rate, and~$K_i$ is the
Michaelis constant. The output $y=x_n$ represents the concentration or
activity of the terminal product.

   Note that \eqref{eq:mm-cascade-internal} 
   is in the form \eqref{eq:cascade-general} with
 \begin{align}\label{eq:enzi+phis} 
 \phi_i(z)
    =
    \frac{V_iz}{K_i+z}.
\end{align} 
 It is straightforward to verify that \eqref{eq:mm-cascade-internal} is a cooperative control system~\cite{monocontsystems}, and that the set $\{x\in\R^ n : x_i \geq 0\} $ is  forward invariant.

The first and second derivatives of~\eqref{eq:enzi+phis} are
\begin{align*}
    \phi_i'(z)
    &=
    \frac{V_iK_i}{(K_i+z)^2},
    \\
    \phi_i''(z)
    &=
    -\frac{2V_iK_i}{(K_i+z)^3}.
\end{align*}
Thus, every $\phi_i$ is strictly increasing and strictly concave on
$\R_{\geq0}$.
   
The following result gives both the sign of GOE and an exact stage-by-stage
decomposition of the loss caused by periodic substrate feeding.

\begin{corollary}[Enzymatic fluctuation penalty]
\label{prop:mm-fluctuation-penalty}
Consider the enzymatic cascade
\eqref{eq:mm-cascade-internal} with a continuous,
nonnegative, $T$-periodic control $v$. Then
\begin{align}\label{eq:mm-exact-goe}
    \GOE(v)
  &  =
    -
    \sum_{k=1}^n
    W_k(v)
    \frac{V_k K_k}{(K_k+ \avg{ \gamma_{k-1}^v})^2}
    \avg{
    \frac{( \gamma_{k-1}^v-\avg{ \gamma_{k-1}^v} )^2}{K_k+ \gamma_{k-1}^v}
    } \\
    &\leq 0\nonumber  ,     
\end{align}
where the~$W_k$s are defined in 
\eqref{eq:cascade-weight},
with 
\begin{equation*}
    L_j(v)
    =
    \frac{V_jK_j}
    {
    (K_j+\avg{\gamma_{j-1}^v})
    (K_j+e_{j-1}^{\avg v})
    }
    >0.
    \label{eq:mm-secant-gain}
\end{equation*}
Furthermore, if $v$ is nonconstant then
\begin{equation}
    \GOE(v)<0.
    \label{eq:mm-goe-strict}
\end{equation}
\end{corollary}

Thus, for this  enzymatic-cascade model, a constant  substrate
feeding produces a larger terminal-product concentration, on average, than
any nonconstant periodic feeding protocol with the same mean substrate concentration.

\begin{proof}
  We begin by computing   the $B_i$s in~\eqref{eq:big_b}. Eq.~\eqref{eq:d_geq} gives
\begin{align*}
D_{\phi_i} (x,z) &= \phi_i(x)-\phi_i(z) -\phi'_i(z)(x-z)\\
& =   -
    \frac{V_iK_i(x-z)^2}
    {(K_i+x)(K_i+z)^2}  . 
\end{align*}
Thus,
\[
 J_i(v)
      =
    \avg{
    D_{\phi_i}
    \bigl(
    \gamma_{i-1}^v,
    \avg{\gamma_{i-1}^v}
    \bigr)
    }=  
    - \frac{ V_iK_i } { (K_i+\avg  {\gamma_{i-1}^v)^2} } 
 \avg{   \frac{ ( \gamma_{i-1}^v-\avg { \gamma_{i-1}^v})^2}
    { K_i+ \gamma_{i-1}^v  }  
}  .
\]
 Note that every~$J_i$ is nonpositive. 
We next compute the   gains in~\eqref{eq:secant-gain}. For any $p,q\geq0$, with $p\neq q$, we have 
\begin{equation*}
    \frac{\phi_i(p)-\phi_i(q)}{p-q}
    =
    \frac{V_iK_i}
    {(K_i+p)(K_i+q)}. 
\end{equation*}
The right-hand side of this equation is well-defined even when $p=q$, 
 so~\eqref{eq:secant-gain} yields 
 \begin{equation*}
    L_i(v)
     = \frac{V_iK_i}
    {(K_i+\avg{\gamma_{i-1}^v})(K_i+e_{i-1}^{\avg {v}})}.
      \end{equation*}
 Theorem~\ref{thm:exact-decomp} yields~\eqref{eq:mm-exact-goe}.
 Since every term on the right-hand side of
\eqref{eq:mm-exact-goe} is nonpositive, we conclude that
   $ \GOE(v)\leq0$.

To prove~\eqref{eq:mm-goe-strict}, suppose that the control $v$ is nonconstant. Then
 $   \avg{
    \frac{(v-\avg v)^2}{K_1+v}
    }
    >0$, 
so $J_1(v)<0$. Since $W_1(v)>0$, at least one term in the sum in~\eqref{eq:mm-exact-goe} is strictly negative,  so
    $\GOE(v)<0$.
\end{proof}

\begin{example}\label{exa:enzi1}
    
We illustrate the enzymatic fluctuation penalty using a specific
example of~\eqref{eq:mm-cascade-internal} with~$n=3$. 
The (arbitrarily chosen) 
parameter values are~$a_i=1$ [$\mathrm{sec}^{-1}$],
$\begin{bmatrix} V_1&V_2&V_3 \end{bmatrix}=
\begin{bmatrix}
    4&2&3
\end{bmatrix}$ [$\mathrm{mM}\,\mathrm{sec}^{-1}$],
and
$\begin{bmatrix} K_1&K_2&K_3 \end{bmatrix}=
\begin{bmatrix}
    2&1&3/2
\end{bmatrix}$ [$\mathrm{mM}$]. 
The periodic substrate input is 
 $
u(t)= 1+0.9 \sin(  t) 
 $  [$\mathrm{mM}$], which is~$T$-periodic with
 $T=  2\pi$  [$\mathrm{sec}$].
Thus, $0.1\leq u(t) \leq 1.9$  for all~$t$. All the simulations were performed using MATLAB. 

The corresponding constant input is~$v(t)=\avg{u}=1$ and the corresponding equilibrium is:
\begin{align}  
e_1^{\avg u} & =\frac{4 \avg {u}}{2+\avg u} = 4/3, \\
e_2^{\avg u} & =\frac{2 e_1^{\avg {u}} }{1+e_1^{\avg u}} =8/7 ,  \\
e_3^{\avg u} & =\frac{3 e_2^{\avg {u}} }{1.5+e_2^{\avg u}} =48/37,   \label{eq:eqa1}
\end{align}
in units of [$\mathrm{mM}$].
Fig.~\ref{fig:limit} 
depicts convergence to the  periodic solution~$\gamma^u$ 
from~$x(0)=0$. The equilibrium $e^ {\avg u} $  corresponding to the constant control is   marked by an~x. 

\begin{figure}[t]
 \centering
 \includegraphics[scale=0.8]{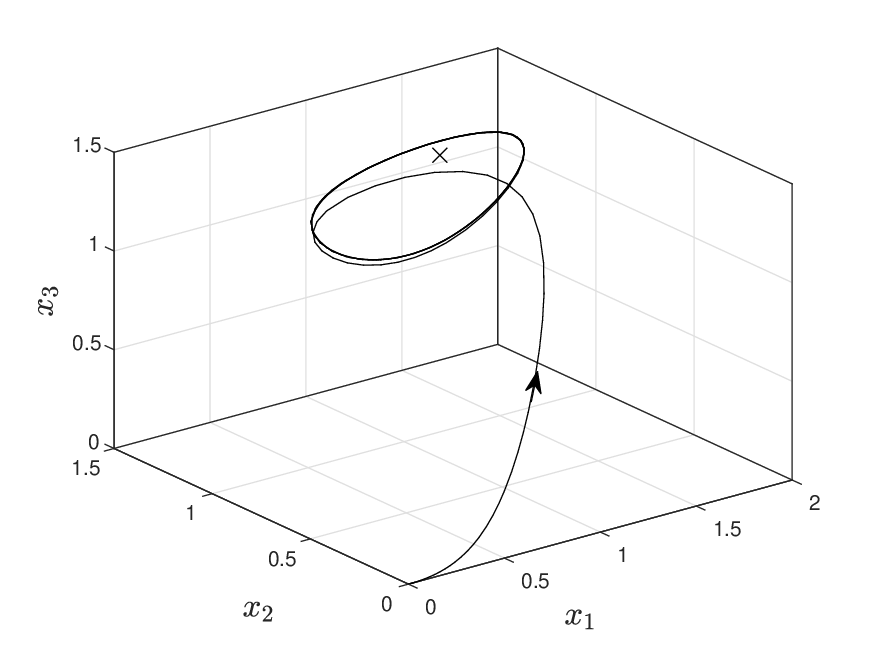}
 \caption{Convergence to the limit cycle~$\gamma^u$ from $x(0)=0$. The equilibrium corresponding to the constant control~$v=\avg u=1$ is marked by an x.}
 \label{fig:limit}
\end{figure}

  Fig.~\ref{fig:peri_waves} depicts the periodic control and its average  and also the coordinates of~$\gamma^u(t)$ vs~$e^{\avg u}$. It may be seen that~$\avg {\gamma_i^u}< e^{\avg u}_i$, $i=1,2,3$. Consequently,
  \[
  \avg {y^\gamma}=\avg{\gamma_3^u}<e^{\avg {u}}_3=y^{\avg u}   ,  
  \]
i.e. a negative GOE. 
A numerical integration gives
$
\frac1{T}\int_0^T \gamma^u_3(t) \mathrm{d}t= 1.2109$,
so
\[
\GOE(u)= 1.2109 - (48/37) =  -0.0864.
\]
Note that the constant control produces  
\[
100 \times \frac{0.0864}{1.2109} =7.14 \%  
\]
more output than the periodic control.

  \begin{figure}[t]
 \centering
 \includegraphics[scale=0.8]{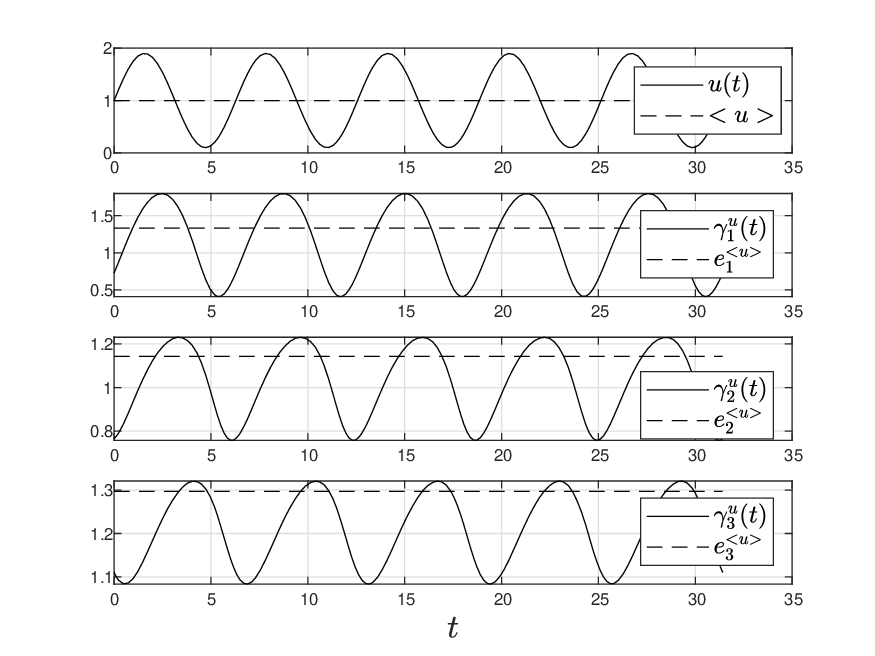}
 \caption{Propagation of the periodic substrate concentration~$u(t)=1+0.9 \sin(t)$ through the
 three-stage Michaelis--Menten cascade.   }
 \label{fig:peri_waves}
\end{figure}

\end{example}

We now use Prop.~\ref{prop:second-order}
to analyze  GOE in the enzymatic cascade~\eqref{eq:mm-cascade-internal} 
for     the perturbed control~\eqref{eq:small-input}.

\begin{corollary}[GOE to second-order in enzymatic cascade]
Consider the enzymatic cascade
\eqref{eq:mm-cascade-internal}   with the control   \eqref{eq:small-input}. Then 
\begin{equation}
    \GOE(v_\varepsilon)
    =
    \frac{\varepsilon^2}{2}
    \sum_{k=1}^n
    \rho_kd_kc_{k-1}
    +
    O(\varepsilon^3),
\end{equation}
where
 \begin{align}\label{eq:rhoke}
     \rho_k
    &=
    a_k^{-1}
    \prod_{j=k+1}^n a_j^{-1} \frac{V_j K_j}{(K_j+e_{j-1}^{\avg u})^2}  ,\nonumber\\
  d_k& = -\frac{2V_k K_k}{(K_k+e_{k-1}^{\avg u})^3}, 
\end{align}
and the $c_k$s are defined in \eqref{eq:defck}.
\end{corollary}

\begin{example} 
Consider the enzymatic cascade  with $n=3$ and parameters as in Example~\ref{exa:enzi1}, and the control
\[
v_\varepsilon(t)=a+\varepsilon \sin(t) \quad [\mathrm{mM}],
\]
with $a=1$. 
The equilibrium $e$ corresponding to the constant control~$\avg{v_\varepsilon}=a=1$ is given in~\eqref{eq:eqa1},
so~\eqref{eq:rhoke} yields 
\begin{align*}
    \rho_1 & =      \frac{V_2 K_2}{(K_2+e_{1}^{\avg u})^2}
       \frac{V_3 K_3}{(K_3+e_{2 }^{\avg u})^2} = (\frac{18}{37})^2  ,\\
 \rho_2 & =     
       \frac{V_3 K_3}{(K_3+e_{2 }^{\avg u})^2} =   \frac{9}{2} (\frac{14} {37}) ^2   ,\\
\rho_3 & =1  ,\\
d_1 & = -\frac{2V_1 K_1}{(K_1+a)^3} = -\frac{16}{27},\\
d_2 & = -\frac{2V_2 K_2}{(K_2+ e_1^{\avg u})^3} = - 4 (\frac{3}{7})^3,\\
d_3 & = -\frac{2V_3 K_3}{(K_3+e^{\avg u}_2)^3} = - 9
(\frac{14}{37})^3.
\end{align*}
We numerically computed $\GOE(v _\varepsilon)$ for different values of $\varepsilon$ 
using  two different  ways. The first is by computing
\begin{align}\label{eq:numcomp1}
\GOE(v _\varepsilon)=\frac1{T}\int_0^ T \gamma ^\varepsilon_3(t) \mathrm{d}t - e_3^{\avg u}. 
\end{align}
The second approach is by  using the second-order approximation 
\begin{align}\label{eq:numcomp2}
\GOE(v _\varepsilon)\approx 
\frac{\varepsilon^2}{2}
    \sum_{k=1}^3
    \rho_kd_kc_{k-1},
\end{align}
 with the $c_k$s computed after numerically integrating the $r _i$s in 
 \eqref{eq:first-general}, 
 with initial condition
  $r_i (0) =\frac{\gamma ^\varepsilon_i(0
  )- e_i^{\avg u}}{\varepsilon}$
   (see \eqref{eq:defrisi}).

  \begin{figure}[t]
 \centering
 \includegraphics[scale=0.8]{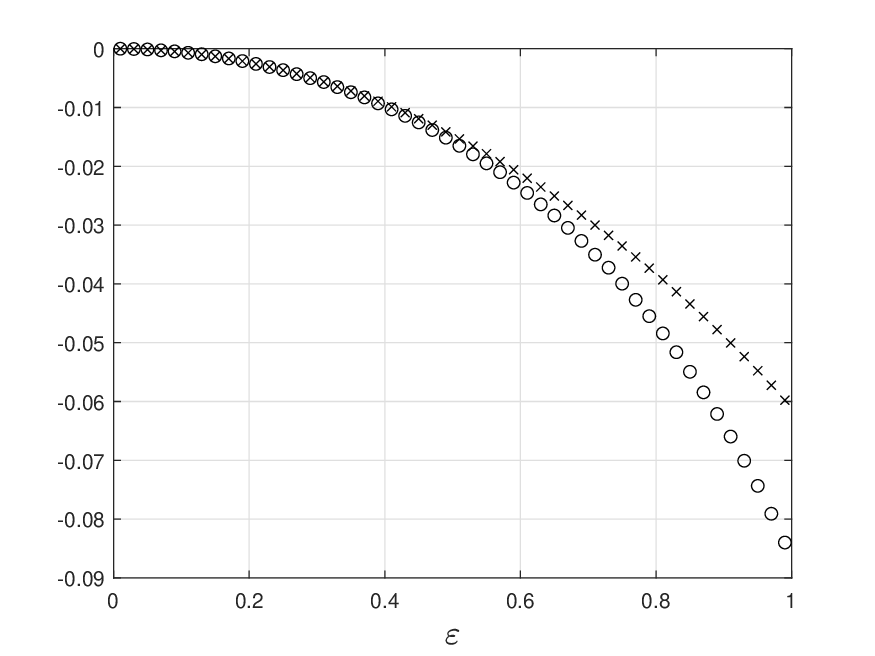}
 \caption{Comparing the values in Eqs. \eqref{eq:numcomp1} 
 (marked by 'o')
 and \eqref{eq:numcomp2}
(marked by 'x') as a function of $\varepsilon \in [0.01,1]$.  }
 \label{fig:epsi}
\end{figure}

 Fig. \ref{fig:epsi} depicts the  values in \eqref{eq:numcomp1} 
 (marked by 'o')
 and \eqref{eq:numcomp2}
(marked by 'x') as a function of~$\varepsilon \in [0.01,1]$.
It may be seen that~$\GOE(v _\varepsilon)$ is always negative, that both graphs are quadratic in~$\varepsilon$,  and that they agree well for small values of $\varepsilon $.  
\end{example}

 \section{Discussion}

This paper considers  GOE in a cascade of scalar linear systems, each with a nonlinear output. The main result  provides an exact decomposition of GOE   into local contributions generated by each nonlinear stage and propagation weights determined by the downstream dynamics. This decomposition separates the creation of average shifts, governed by the curvature of the nonlinear maps through Jensen gaps, from their transmission through the cascade. This provides a transparent explanation of why the sign of the overall GOE depends not only on local convexity or concavity,  but also on the monotonicity of downstream interconnections.
  A  Bregman-divergence representation interprets GOE as the accumulation of local geometric discrepancies.

  The paper also provides a
   second-order expansion of GOE
   that identifies the distinct roles of local curvature, fluctuation energy, and downstream differential gains.

Topics for further research include: 
(1) proving monotonicity of GOE with respect to insertion/removal of stages;
(2) deriving bounds on GOE;
(3) optimization of periodic excitation waveforms; 
and (4) developing guidelines 
for optimal ordering of nonlinearities along the  cascade.

 We hope that the stage-by-stage viewpoint introduced here will prove useful in the analysis,  design, and optimization  of  other periodically forced nonlinear systems.

\subsection*{Acknowledgments}
RM gratefully acknowledges Mrs. Anna Charps, his lecturer during his practical engineering studies, for introducing him to cascade systems in her Industrial Automation course. Her inspiring lectures sparked an early interest in these systems and ultimately helped motivate the ideas developed in this work.

\subsection*{ Declaration of generative AI and AI-assisted technologies in the manuscript preparation process}
The authors used ChatGpt  for editing and proofreading. After using ChatGpt, the authors  reviewed and edited the content as needed and take full responsibility for the content of the published article.

\end{document}